\documentclass[a4paper,12pt]{amsart}
\usepackage{fullpage}
\usepackage{enumerate}
\usepackage{amssymb}
\usepackage{amsmath}
\usepackage{latexsym}
\usepackage{amsthm}
\usepackage{xypic}
\usepackage{xcolor}
\usepackage[english]{babel}
\usepackage{hyperref}

\newtheorem{theorem}{Theorem}[section]
\newtheorem{lemma}[theorem]{Lemma}
\newtheorem{proposition}[theorem]{Proposition}
\newtheorem{cor}[theorem]{Corollary}
\theoremstyle{definition}

\theoremstyle{remark}
\newtheorem{remark}[theorem]{Remark}

\theoremstyle{question}

\newtheorem*{theorem*}{Theorem}
\numberwithin{equation}{section}

\newcommand{\Q}{\mathbb{Q}}

\newcommand{\red}{}

\begin{document}

\title{On the Northcott property and local degrees}
\author{S. Checcoli}
\address{S.~Checcoli, Institut Fourier, Universit\'e Grenoble Alpes, 100 rue des Math\'ematiques,  38610 Gi\`eres, France}
\email{sara.checcoli@univ-grenoble-alpes.fr}

\author{A. Fehm}
\address{A.~Fehm, Institut f\"ur Algebra, Fakult\"at Mathematik, Technische Universit\"at Dresden, 01062 Dresden, Germany}
\email{arno.fehm@tu-dresden.de}

\begin{abstract} 
We construct infinite Galois extensions $K$ of $\mathbb{Q}$
that satisfy the Northcott property on elements of small height,
and where this property can be deduced solely from the splitting behavior of prime numbers in $K$.
We also give examples of Galois extensions of $\mathbb{Q}$
which have finite local degree at all prime numbers
and do not satisfy the Northcott property.
\end{abstract}
\maketitle

\section{Introduction}

\noindent
Let $\overline{\Q}$ be the field of algebraic numbers and denote by $h:\overline{\Q}\rightarrow \mathbb{R}_{\geq 0}$ the (absolute logarithmic) Weil height.  The set of algebraic numbers of height zero is completely described by a theorem of Kronecker and consists of 0 and all roots of unity, but there are many interesting open problems concerning algebraic numbers of non-zero small height.
Two important statements emerge in this context.
The first, Northcott's theorem, ensures that a set of algebraic numbers whose elements have both their height and their degree bounded is finite.
The second is a famous conjecture of Lehmer, which states that for every algebraic number the product of its height and its degree is either 0 or bigger than an absolute positive constant. This conjecture is still open in general, but has been proved for many classes of algebraic numbers. Another celebrated related problem is the Schinzel-Zassenhaus conjecture on a lower bound for the maximal absolute value of the conjugates of an algebraic integer that is not a root of unity, in terms of its degree,
proved very recently by Dimitrov \cite{Dim}.
Now one could ask in which cases the above statements are still true if one ``forgets the degree''.
More precisely, following Bombieri and Zannier \cite{BZ}, 
a set $L$ of algebraic numbers has the \emph{Northcott property (N)} 
if for every $T>0$ the set of $\alpha\in L$ with $h(\alpha)<T$ is finite, 
and it has the \emph{Bogomolov property (B)} if 
there exists $T>0$ for which there is no $\alpha\in L$ with
$0<h(\alpha)<T$.
It is easy to see that property (N) implies property (B). 

A particularly interesting case is when $L$ is a field.
By Northcott's theorem, both properties hold for number fields, but deciding the validity of these properties for infinite algebraic extensions of $\mathbb{Q}$ is, in general, a difficult problem which has been studied by many authors. 
There has been a lot of work around property (B) in the last years,
among them 
\cite{Sch,AD,AZ1,BZ,DZ,AZ2,Wid,Hab,CW,ADZ,Pot,Pot2,Gal,Frey, Feh, Plessis}.
On the other hand, results on property (N) are quite rare. The first result is 
 \cite[Theorem 1]{BZ}, where the authors considered the compositum $K^{(d)}$ of all extensions of degree at most $d$ of a number field $K$  and proved that property (N) is satisfied by $K^{(d)}_{\rm ab}$, the maximal subextension of $K^{(d)}$ which is abelian over $K$. Some variants of this result were later considered in \cite{CW},
and it was applied to obtain undecidability results in \cite{VV}.

In \cite{BZ}, the authors also studied properties (N) and (B) for 
Galois extensions $L/\Q$ in relation to certain local properties of $L$. Recall that $L$ has \emph{finite local degree}
at a prime number $p$ if the completion $L_v$ of $L$
with respect to \red{any} valuation $v$ extending the $p$-adic valuation on $\mathbb{Q}$ is a finite extension of $\mathbb{Q}_p$.
Bombieri and Zannier proved the following  (see \cite[Theorem 2]{BZ}):

\begin{theorem}[Bombieri--Zannier]
Let $L/\Q$ be a Galois extension and let $S(L)$ be the set of 
prime numbers $p$ for which $L$ has finite local degree at $p$. If $S(L)$ is non-empty, then $L$ has property (B). More precisely, setting 
\begin{equation}\label{BZ-Bogo}
 \mathfrak{S}(L):=\sum_{p\in S(L)}\frac{\log p}{e_p(p^{f_p}+1)}
\end{equation} 
where $e_p$ and $f_p$ denote the ramification index and the inertia degree of $L$ at $p$, respectively, one has
\[\liminf_{\alpha\in L} h(\alpha)\geq \red{\frac{1}{2}\mathfrak{S}(L)}.\]
\end{theorem}

They also remarked that, if \red{the sum $\mathfrak{S}(L)$ diverges}, the field $L$ has property (N), but they expected this to be the case only for finite extensions of $\Q$ (for which, as already pointed out, property (N) holds).

The first result of this paper is to show that \red{the sum $\mathfrak{S}(L)$ might diverge} also for infinite Galois extensions \red{$L/\Q$, therefore providing} a non-trivial criterion to ensure property (N). In addition there is a certain flexibility in the choice of the Galois groups of such extensions. More precisely, we prove:

\begin{theorem}\label{main-div}
Let $(G_i)_{i\in\mathbb{N}}$ be a family of finite solvable groups. Then there exists a totally real Galois extension $L/\Q$ 
with $\mathrm{Gal}(L/\Q)=\prod_{i\in\mathbb{N}} G_i$\red{, the direct product of the groups $G_i$, such that $L$}
 has finite local degree at all prime numbers, and $\mathfrak{S}(L)=\infty$.
\end{theorem}

The proof of this theorem is carried out in Section \ref{proof-main-div}. Our construction is elementary, but it requires a sharper version Shafarevich's theorem on the realisability of solvable groups with splitting conditions, namely the fact that every finite solvable group can be realised over $\Q$ by a Galois extension which is totally split at all prime numbers in a given finite set (Theorem \ref{sha-split}), which we prove in Appendix \ref{shafarevich}.
Theorem \ref{main-div} 
leads to undecidability results (Remark \ref{rem-undecidability})
and also holds for certain non-solvable groups $G_i$ (Remark \ref{rem-Gru}).

As an easy corollary, we recover a result already proved  in \cite[Theorem 4]{CW}:
\begin{cor} Every direct product $\prod_{i\in\mathbb{N}}G_i$ of finite solvable groups can be realised over $\Q$ by a Galois extension with property (N).
\end{cor}
The proof of this result in \cite{CW} was quite different and employed a  criterion of Widmer \cite[Theorem 3]{Wid}, which, after the work of Bombieri and Zannier, has been essentially the only other known result concerning property (N). His criterion is based  on the growth of certain discriminants in infinite towers of number fields. We recall it here:

\begin{theorem}[Widmer]\label{thm-wid} 
Let $K_0\subsetneq K_1\subsetneq\ldots$ be a tower of number fields. Suppose that
\begin{equation}\label{eq-wid}\inf_{K_{i-1}\subsetneq M\subseteq K_i} \left(N_{K_{i-1}/\Q}\left(D_{M/K_{i-1}}\right)\right)^{\frac{1}{[M:K_0][M:K_{i-1}]}}\rightarrow \infty\; \text{ as }\; i\rightarrow \infty\end{equation}
where $M$ varies among the intermediate fields of $K_{i}$ strictly containing $K_{i-1}$, $D_{M/K_{i-1}}$ denotes the relative discriminant \red{of the extension $M/K_{i-1}$} and $N_{K_{i-1}/\Q}\left(D_{M/K_{i-1}}\right)$ denotes the unique positive integer generating the \red{norm} ideal. Then the field $L=\bigcup_{i\geq0} K_i$ has property (N).
\end{theorem}

As a second result of this paper, we show the existence of infinite Galois extensions $L/\Q$ with property (N) that satisfy the conclusion of Theorem \ref{main-div} (i.e. having $\mathfrak{S}(L)=\infty$), but to which neither Widmer's criterion nor the result of Bombieri and Zannier on $K^{(d)}_{\rm ab}$ applies.
 More precisely, in Section \ref{sect-comp-widmer} we prove the following:
\begin{theorem}\label{thm-comp-widmer}
There exists an infinite pro-cyclic Galois extension $L/\Q$ with  $\mathfrak{S}(L)=\infty$, but such that every tower of number fields $\Q=K_0\subsetneq K_1\subsetneq \ldots$ with $L=\bigcup_{i\geq0}K_i$  does not satisfy condition \eqref{eq-wid} of \red{Theorem \ref{thm-wid}.}
\end{theorem}

The third result of this paper is again on the relation, for a Galois extension $L/\Q$, between its local degrees and property (N).
One of the ingredients used in the proof of property (N) for $K^{(d)}_{\rm ab}$ is the fact that the local degrees of the field $K^{(d)}$ are uniformly bounded \red{at all prime numbers}, and in \cite{BZ} it is asked whether it is true that property (N) holds for  all fields with this local property. This seems to be out of reach for the moment: indeed it is still open for the field $K^{(d)}$ itself when $d\geq 3$ (as $K^{(2)}/K$ is abelian), even in the simplest case $K=\Q$ and $d=3$. 

Then one could ask the more modest \red{question} of whether property (N) holds for all extensions with finite
(but not necessarily uniformly bounded)
 local degrees. In \cite[Proposition 1.3]{Feh} the author constructed extensions having local degree one at infinitely many prime numbers \red{(but also infinite local degree at infinitely many other prime numbers)} without property (N). We prove the existence of extensions with finite local degree at every prime number
which do not have property (N):
\begin{theorem}\label{main-bld}
There exists an infinite Galois extension $L/\Q$ which has finite local degree at all prime numbers, but does not have property (N).
\end{theorem}
Theorem \ref{main-bld} is proved in Section \ref{proof-main-bld}
where it is deduced from a theorem of Fili \red{\cite[Theorem 2]{Fil} (see also Theorem \ref{Filthm})} regarding heights in fields of totally $S$-adic numbers,
generalizing a result of Bombieri and Zannier \cite[Example 2]{BZ}.


\section{Proof of Theorem \ref{main-div}}\label{proof-main-div}

\noindent
Let $L/\Q$ be a (possibly infinite) Galois extension and let $S(L)$ be the set of prime numbers $p$ such that $L$ has \red{finite} local degree at $p$. For every $p\in S(L)$ denote by $e_p$ and $f_p$ the ramification index and inertia degree of $L$ at $p$, respectively. As already recalled, in \cite[Theorem 2]{BZ} Bombieri and Zannier studied, when $S(L)$ is not empty, the quantity \[\mathfrak{S}(L)=\sum_{p \in S(L)}\frac{\log p}{e_p(p^{f_p}+1)}\] in connection with property (B). 
They also remarked that if $\mathfrak{S}(L)=\infty$, then $L$ has property (N). This is the case when the extension $L/\Q$ is finite, as follows from Chebotarev's density theorem and the prime number theorem. We repeat the proof of this fact here for the convenience of the reader.
\begin{proposition}\label{div-nf}
Let $L/\Q$ be a finite Galois extension. Then  $\mathfrak{S}(L)=\infty$.
\end{proposition}
\begin{proof}
Denote by $T(L)$ the set of prime numbers that split totally in $L$ and, \red{for $k\geq 1$,} set \[a_k=\sum_{\substack{p \in T(L)\\ 2^k<p\leq 2^{k+1}}}\frac{\log p}{p}.\] Then
 \begin{equation}\label{divs}\mathfrak{S}(L)=\sum_{p \text{ prime }}\frac{\log p}{e_p(p^{f_p}+1)}\geq \red{\sum_{p \in T(L)}\frac{\log p}{p+1}}\geq \frac{1}{2} \lim_{N\rightarrow \infty}\sum_{k=0}^{N} a_k.\end{equation}
Notice that \[a_k\geq \sum_{\substack{p \in T(L)\\ 2^k<p\leq 2^{k+1}}}\frac{\log (2^{k+1})}{2^{k+1}}\]
and $|\{p\in T(L)\mid 2^k<p\leq 2^{k+1}\}|=|\{p\in T(L)\mid p\leq 2^{k+1}\}|-|\{p\in T(L)\mid p\leq 2^{k}\}|$.
By Chebotarev's density theorem and the prime number theorem, \red{for big enough values of $k$}, we have
\begin{align*} a_k &\geq \frac{1}{[L:\Q]} \left(\frac{3}{4}\cdot\frac{2^{k+1}}{\log(2^{k+1})}-\frac{4}{3}\cdot\frac{2^{k}}{\log(2^{k})}\right)\frac{\log(2^{k+1})}{2^{k+1}} \\ &= \frac{1}{[L:\Q]}\left(\frac{3}{4}-\frac{2(k+1)}{3k}\right)>\frac{1}{13[L:\Q]}
\end{align*}
and we can conclude by \eqref{divs}.
\end{proof}
To prove Theorem \ref{main-div} we need to construct infinite Galois extensions $L/\Q$ for which $\mathfrak{S}(L)=\infty$ and having as Galois group any prescribed direct product $\prod_{i\in\mathbb{N}}G_i$ of finite solvable groups. The proof uses the following refined version of Shafarevich's \red{theorem} on the realisability of finite solvable groups as Galois groups over $\Q$ with splitting conditions,
which is claimed in \cite[Exercise (a), pag. 597]{NSW}:

\begin{theorem}\label{sha-split}
Let $G$ be a finite solvable group and $S$ a finite set of primes of a number field $K$. 
Then there exists a Galois \red{extension} of $K$ with Galois group $G$ in which all primes in $S$ split totally.
\end{theorem}

\red{Due} to \red{the} lack of a written proof \red{of this result} (and the fact that special cases of \red{it} were published after the book appeared, like \cite{KM})
we sketch a proof in Appendix \ref{shafarevich}.
 
\begin{cor} \label{cor:sha}
Let $G$ be a finite solvable group, $S$ a finite set of prime numbers, and $K$ a number field. 
Then there exists a totally real Galois extensions $L$ of $\Q$ with Galois group $G$ in which all prime numbers in $S$ split totally 
and which is linearly disjoint from $K$ over $\Q$.
\end{cor}

\begin{proof}
Theorem \ref{sha-split} allows, for every $n$, to realise $G^n$ over $\Q$ by an extension  in which all primes in $S\cup\{\infty\}$ split totally.
By Galois theory, 
for $n$ large enough, at least one of the factors will correspond to a Galois extension $L/\Q$ with Galois group $G$
linearly disjoint from $K/\Q$.
\end{proof}


\begin{proof}[Proof of Theorem \ref{main-div}]
Let $(G_i)_{i\geq 1}$ be a family of finite solvable groups.
For a Galois extension $K/\Q$
which has finite local degree at the prime number $p$
we write
\begin{equation}\label{betap}
 \mathfrak{S}_p(K)=\frac{\log p}{e_p(p^{f_p}+1)} 
\end{equation}
where $e_p$ and $f_p$ are the ramification index and the inertia degree of $K$ at $p$.

Set $n_0=1$ and $L_0=\Q$.
We construct inductively an increasing sequence of integers 
$(n_i)_{i\geq1}$ and
a sequence of totally real number fields $(F_i)_{i\geq1}$
such that, with $L_i=F_1\cdots F_i$ the compositum, for every $i\geq1$ we have
\begin{enumerate}[(a)]
\item\label{cond-c} $\sum_{n_{i-1}\leq p<n_{i}} \mathfrak{S}_p(L_{i-1})\geq 1$, 
\item\label{cond-a}  every prime number $p\leq n_i$ is totally split in $F_i$, and
\item\label{cond-b} $F_i/\Q$ is Galois with $\mathrm{Gal}(F_i/\Q)=G_i$
and linearly disjoint from $L_{i-1}/\Q$.
\end{enumerate}
Suppose we have already constructed $n_1,\ldots,n_{i-1}$ and $F_1,\dots,F_{i-1}$.
As $L_{i-1}=F_1\cdots F_{i-1}$ is a number field, by Proposition \ref{div-nf} we have that $\sum_{p}\mathfrak{S}_p(L_{i-1})=\infty$. Thus there exists $n_i$ such that condition  \eqref{cond-c} holds.
Now, by Corollary \ref{cor:sha}, there exists a totally real number field $F_i$ linearly disjoint from $L_{i-1}/\Q$ 
such that $F_i/\Q$ is Galois of group $G_i$ 
and all prime numbers $p\leq n_i$ split totally in $F_i$,
thus satisfying \eqref{cond-a} and \eqref{cond-b}.

Then the field $L=\bigcup_{i\geq0}L_i=\prod_{i\geq1}F_i$ satisfies the claim. Indeed, by condition \eqref{cond-b}, it is an Galois extension of $\Q$ with Galois group the direct product $\prod_{i\geq1}G_i$. Moreover, 
\[\mathfrak{S}(L)=\sum_{i=1}^{\infty}\sum_{n_{i-1}\leq p<n_i} \mathfrak{S}_p(L)
\stackrel{\eqref{cond-a}}{=}\sum_{i=1}^{\infty}\sum_{n_{i-1}\leq p<n_i} \mathfrak{S}_p(L_{i-1})\stackrel{\eqref{cond-c}}{\geq}\sum_{i=1}^{\infty}1=\infty.\]
\end{proof}

\begin{remark}\label{rem-undecidability}
As the field $L$ constructed in Theorem \ref{main-div} is totally real and has property (N),
this immediately leads to the undecidability of the first-order theory of the ring of integers $\mathcal{O}_L$ of $L$ by \cite[Theorem 2]{VV}. 
By a suitable choice of the groups $G_i$ one can then obtain the undecidability of the first-order theory of $L$ itself,
using for example \cite{Videla} or \cite{Shlapentokh}.
\end{remark}

\begin{remark}\label{rem-Gru}
We remark that in the statement of Theorem \ref{main-div}, the groups in the family $(G_i)_{i\in \mathbb{N}}$ can be chosen to be any finite groups such that, for every \red{$i$}, and for every finite set of prime numbers $S$, \red{$G_i$} has infinitely many pairwise linearly disjoint realisations over $\Q$ in which all prime numbers in $S$ split totally. For instance, if one allows the extension $L/\Q$ to have \red{finite} local degree at all but finitely many prime numbers, one can take groups for which certain Grunwald problems, about the existence of a Galois extension of $\Q$ with Galois group and given completion at a finite set of prime numbers,  have many solutions. In particular, by a result of D\`ebes and Ghazi \cite{DG}, these include  certain non-solvable groups. 
\end{remark}

As already recalled, given a number field $K$ and an integer $d\geq 1$, in \cite{BZ} the authors considered the compositum $K^{(d)}$  of all extensions of $K$ of degree at most $d$. In \cite[Theorem 1]{BZ} they showed that property (N) holds for the maximal subextension of $K^{(d)}$ which is abelian over $K$, denoted $K^{(d)}_{\rm ab}$. 
This implies that the field $K^{(2)}=K^{(2)}_{\rm ab}$ has property (N), while, for $d\geq 3$, deciding property (N) for $K^{(d)}$ is an open problem, even for the field $\Q^{(3)}$.

Theorem \ref{main-div} implies that the divergence of \red{the sum $\mathfrak{S}(L)$} is actually a criterion that one can use to prove property (N) for \red{an infinite Galois extension $L$ of $\Q$.} However we note that this criterion does not apply to the fields $\Q^{(d)}$ and not even to $\Q^{(d)}_{\rm ab}$, as the following observation shows. 
\begin{proposition}\label{propQd} 
For every $d\geq 2$, $\mathfrak{S}(\Q^{(d)}_{\rm ab})<\infty$, 
in particular $\mathfrak{S}(\Q^{(d)})<\infty$.
\end{proposition}
\begin{proof}
For every prime number $p$, $\Q^{(2)}$ contains quadratic extensions of $\Q$ in which $p$ is inert (take for instance $\Q(\sqrt{n_p})$ where $n_p$ is any square-free integer which is not a square modulo $p$).
\red{Thus, for every $p$, the inertia degree of $\Q^{(2)}$ at $p$ is at least $2$ and} \[\mathfrak{S}(\Q^{(2)})\leq \sum_{p \text{ prime }}\frac{\log p}{p^2+1}\leq \sum_{n=1}^{\infty}\frac{\log n}{n^2}<\infty.\]
\red{As} $\Q^{(2)}=\Q^{(2)}_{\rm ab}\subseteq\Q^{(d)}_{\rm ab}\subseteq\Q^{(d)}$, \red{we have}  $\mathfrak{S}(\Q^{(d)})\leq\mathfrak{S}(\Q^{(d)}_{\rm ab})\leq\mathfrak{S}(\Q^{(2)})<\infty$.
\end{proof}
\begin{remark}
Note however that for $d\geq3$, in view of Theorem \ref{main-div}, $\Q^{(d)}$ possesses both abelian and non-abelian subfields $L$ of infinite degree over $\Q$ for which $\mathfrak{S}(L)=\infty$. Indeed, it is sufficient to choose the groups \red{$G_i$} in Theorem \ref{main-div} to be cyclic of order 2 to obtain $L\subseteq \Q^{(2)}$, or to choose  all the \red{$G_i$'s} equal to the symmetric group on 3 elements to obtain $L\subseteq \Q^{(3)}$, but $L\not\subseteq \Q^{(3)}_{\rm ab}$.

Proposition \ref{propQd} also shows that, as one could expect, the divergence of $\mathfrak{S}(L)$ is not equivalent to property (N) for Galois extensions $L/\Q$ with \red{finite} local degrees.
\end{remark}
\begin{remark}
It was proved in \cite[Theorem 2.1]{DZ} that property (N) is preserved under finite extensions. It would be interesting to \red{study, in this setting, the behavior of the function $\mathfrak{S}(\cdot)$. More precisely, let $L/\Q$ be a Galois extension such that \red{$\mathfrak{S}(L)=\infty$} and let $F/L$ be a finite extension with $F/\Q$ Galois. Is it true that $\mathfrak{S}(F)=\infty$?}

\end{remark}

\section{A comparison with Widmer's criterion: proof of Theorem \ref{thm-comp-widmer}}\label{sect-comp-widmer}

\noindent
As recalled in the Introduction, in \cite[Theorem 3]{Wid} Widmer gave a criterion for an infinite algebraic extension of $\Q$ to have property (N) which is based  on the growth of the discriminants of its finite subextensions (Theorem \ref{thm-wid}). 
The aim of this section is to compare his criterion to the criterion for property (N) provided by Theorem \ref{main-div}. More precisely, we show that there are infinite Galois extensions  $L/\Q$ which fail Widmer's criterion, but for which $\mathfrak{S}(L)=\infty$, proving Theorem \ref{thm-comp-widmer} from the Introduction.
We will need the following:
\begin{lemma}\label{lem-comp-widmer} \red{For every integer $N\geq 1$, there exists a constant $C(N)>0$ with the following property: 
For every set of prime numbers $S$ of cardinality $N$ and for every prime $p\geq C(N)$,} there exists a cyclic Galois extension $F/\Q$ of degree $p$ such that
\begin{enumerate}[(i)]
\item all prime numbers in $S$ split totally in $F$, and
\item \red{$|{\Delta_{F}}|^{1/p^2}\leq 3$},
\end{enumerate}
\red{where by $\Delta_{K}$ we denote the absolute discriminant of the number field $K$.}
\end{lemma}
\begin{proof}
We first recall a classical result on the distribution of prime numbers in arithmetic progressions. For $x>0$ a real number and $p$ a prime number, let 
\[
 \pi(x,p,1) = |\{\ell\leq x:\ell\equiv 1\mbox{ mod } p, \ell \text{ prime}\}|.
\]
By a result of Walfisz (see \cite{wal} or \cite[Corollary 5.29]{IK}), there exists a  constant $T>0$ such that, for all $x\geq 3$ and all $p$ \red{one has}
\[
 \left|\pi(x,p,1)-\frac{{\rm Li}(x)}{p-1}\right|\leq T\frac{x}{(\log x)^3}
\]
where ${\rm Li}(x)=\int_{2}^x \frac{dt}{\log t}$ is the logarithmic integral. 
By the series representation of ${\rm Li}(x)$ one has ${\rm Li}(x)\geq \frac{x}{\log x}$ for all \red{$x\geq 7$.} 
In particular we have that, for all \red{$x\geq 7$}

\begin{equation}\label{walf}
\pi(x,p,1)\geq \frac{x}{(p-1)\log x }-T \frac{x}{(\log x)^3}.
\end{equation}

Now, let $S$  be a set of \red{$N$} prime numbers, let $p$ be a prime number such that 
\begin{equation}\label{c1-p}
p\geq \red{C(N):=\max\{121,(2T+1)(N+1)^2\}}
\end{equation}
\red{and set $x_0=e^{{p}/{(N+1)}}$.}
Then by equation \eqref{walf} we have
\begin{equation}\label{eq-pi}
 \pi(x_0,p,1)\geq \frac{e^{\frac{p}{N+1}}(N+1)}{p^2}\left(1-T\frac{(N+1)^2}{p}\right)> \frac{e^{\frac{p}{N+1}}}{2 p^2}(N+1).\end{equation}
\red{By condition \eqref{c1-p}, as $p\geq 121$,} then \red{$4 p^2\leq e^{p^{1/2}}\leq{e^{\frac{p}{N+1}}}$.} \red{Thus \eqref{eq-pi} gives} 
\[
 \pi(x_0,p,1)\geq 2(N+1).
\]
\red{Therefore} there exist prime numbers $\ell_1,\ldots,\ell_{N+1}$ not in $S\cup\{p\}$ with $\ell_j\leq x_0$ and $\ell_j\equiv 1 \mod p$
 for all $1\leq j\leq N+1$.

We now proceed as in the proof of \cite[Theorem 6]{Che19}. \red{Denoting by $\zeta_{\ell_i}$ a primitive $\ell_i$-th root of unity,} we first remark that each field  $\Q(\zeta_{\ell_i})$ contains a cyclic extension $M_i/\Q$ of degree $p$. \red{Then the compositum of all fields $M_1,\ldots, M_{N+1}$  is an elementary $p$-abelian extension of $\Q$ of degree $p^{N+1}$. Its subfield fixed by all elements $\{\sigma_q \mid q\in S\}$, where $\sigma_q$ is the Frobenius at $q$, is an abelian extension of $\Q$ of degree at least $p$.} Thus it contains a \red{subextension} $F/\Q$ which is cyclic of degree $p$, totally split at all prime numbers in $S$ and unramified outside $\ell_1,\ldots, \ell_{N+1}$ \red{(hence, in particular, unramified at $p$).} 

By \red{the discriminant-conductor formula (see \cite[Chap. VII, \S 11, Formula (11.9)]{Neukirch}) and the structure of the conductor of a cyclic extension of $\Q$ of degree $p$ unramified at $p$ (see \cite[p.~831]{May92})} we have 
\red{\[|\Delta_{F}|\leq (\ell_1\cdots \ell_{N+1})^{p-1},\]}
thus \red{$|\Delta_{F}| \leq x_0^{(N+1)(p-1)}$. As $x_0=e^{p/(N+1)}$,} we get
\red{\[|\Delta_{F}|^{\frac{1}{p^2}}\leq 3.\]}
\end{proof}

 \red{Before proceeding to the proof of Theorem \ref{thm-comp-widmer} we recall some notation, which is the same as in  \cite[pp.~2-3]{Wid}.  For a finite extension of number fields $M/K$, $D_{M/K}$ denotes  the relative discriminant, which is the discriminant ideal of the ring of integers $\mathcal{O}_M$ relative to $K$; in particular, $D_{M/\Q}$ is the ideal generated by the absolute discriminant $\Delta_M$ of $M/\Q$.  We also denote by $N_{K/\Q}(\cdot)$ the norm from $K$ to $\Q$ and by $N_{K/\Q}\left(D_{M/K}\right)$ the unique positive integer generating this ideal. In particular, by \cite[Chap.II, \S 2, Corollary (2.10)]{Neukirch} we have \begin{equation}\label{norm-disc} N_{K/\Q}\left(D_{M/K}\right)=\frac{|\Delta_M|}{|\Delta_K|^{[M:K]}}. \end{equation}
We also recall that, if $F$ and $K$ are linearly disjoint number fields with coprime discriminants, then the discriminant of their compositum $KF$ is given by
\begin{equation}\label{disc-comp}
\Delta_{KF}=\Delta_K^{[F:\Q]}\Delta_F^{[K:\Q]},
\end{equation}
cf.~\cite[Theorem 4.26]{Nark}.
We are now ready to prove Theorem \ref{thm-comp-widmer}.}
\begin{proof}[Proof of Theorem \ref{thm-comp-widmer}]
\red{We argue} as in the proof of Theorem \ref{main-div}. 
Set $n_0=1$ and $L_0=\Q$. We construct inductively an increasing sequence of integers $(n_i)_{i\geq 1}$,
a sequence of pairwise distinct prime numbers $(p_i)_{i\geq 1}$
and a sequence $(F_i)_{i\geq1}$ of cyclic Galois extensions of $\Q$
with $[F_i:\Q]=p_i$ such that,
with $L_i=F_1\cdots F_i$ the compositum,
for every $i\geq1$ we have
\begin{enumerate}[(a)]
\item\label{cc2} $\sum_{n_{i-1}\leq p<n_{i}} \mathfrak{S}_p(L_{i-1})\geq 1$, where $\mathfrak{S}_p$ is defined in \eqref{betap},
\item\label{cc1} $F_i$ is totally split at all prime numbers in the set $$
 S_i=\{p : p\leq n_i\}\cup\bigcup_{j=1}^i\{p:p|\Delta_{F_j}\},
$$
\item\label{cc3} $|\Delta_{F_i}|^{1/p_i^2}\leq 3$.
\end{enumerate}
Suppose we have already constructed $n_1,\ldots,n_{i-1}$, $p_1,\ldots,p_{i-1}$ and $F_1,\ldots,F_{i-1}$. Then, as $L_{i-1}=F_1\cdots F_{i-1}$ is a number field and thus, by Proposition \ref{div-nf}, $\sum_p \mathfrak{S}_p(L_{i-1})=\infty$, there exists $n_i$ satisfying condition \eqref{cc2}. Let $N_i$ be the cardinality of the set $S_i$ defined in \eqref{cc1}. By Lemma \ref{lem-comp-widmer}, if we choose $p_i>C(N_i)$, there exists a Galois extension $F_i/\Q$ which is cyclic of degree $p_i$ and satisfies \eqref{cc1} and \eqref{cc3}.
Note that condition \eqref{cc1} also guarantees \red{\[\gcd(\Delta_{F_i},\Delta_{F_1}\cdots \Delta_{F_{i-1}})=1.\]}
Let $L=\bigcup_{i\geq0}L_i=\prod_{i\geq1}F_i$. Then, as in the proof of Theorem \ref{main-div}, conditions \eqref{cc1} and \eqref{cc2} imply that $\mathfrak{S}(L)=\infty$. 

We now want to prove that given any tower of number fields 
$\Q=K_0\subsetneq K_1\subsetneq\ldots$ such that $L=\bigcup_{i\geq0} K_i$, we have
$$
 \inf_{K_{i-1}\subsetneq M\subseteq K_i} \left(N_{K_{i-1}/\Q}\left(D_{M/K_{i-1}}\right)\right)^{\frac{1}{[M:K_0][M:K_{i-1}]}}
\not\rightarrow \infty\; \text{ as } i\rightarrow \infty,
$$
where $M$ varies among the intermediate fields of $K_{i}$ strictly containing $K_{i-1}$.

We first remark that $[K_i:\Q]$ equals a product of distinct prime numbers, as $K_i$ is a number field contained in $L$. Letting  $D_i=\{j\in \mathbb{N} : p_j\mid [K_i:\Q]\}$, one sees that $K_i$ must be equal to the compositum of  the  fields $\{F_{j}\mid j\in D_i\}$. 
We also note that if $M$ is a subfield of $K_{i}$ strictly containing $K_{i-1}$, then $M$ is the compositum of $K_{i-1}$ and a linearly disjoint field $M'$, which is  \red{the} compositum of \red{some} fields in the family $(F_{j})_{j\in D_{i}\setminus D_{i-1}}$.
\red{By \eqref{norm-disc} we have 
\[N_{K_{i-1}/\Q}\left(D_{M/K_{i-1}}\right)=\frac{|\Delta_M|}{{|\Delta_{K_{i-1}}|}^{[M:K_{i-1}]}}.\]}
As the discriminants of the fields $K_{i-1}$ and $M'$ are coprime, \red{by formula \eqref{disc-comp} we have
\[
\Delta_M=\Delta_{K_{i-1}}^{[M':\Q]}\Delta_{M'}^{[K_{i-1}:\Q]}
\]}
and, as $[M':\Q]=[M:K_{i-1}]$, \red{we  obtain \[N_{K_{i-1}/\Q}\left(D_{M/K_{i-1}}\right)=|\Delta_{M'}|^{[K_{i-1}:\Q]}.\]}
Therefore \[\inf_{K_{i-1}\subsetneq M\subseteq K_i}  \left(N_{K_{i-1}/\Q}\left(D_{M/K_{i-1}}\right)\right)^{\frac{1}{[M:K_0][M:K_{i-1}]}}=\red{\inf_{M'}{|\Delta_{M'}|}}^{\frac{1}{[M':\Q]^2}},\]
where $M'$ varies among all nontrivial composita of  fields in the family $(F_j)_{ j\in D_{i}\setminus D_{i-1}}$. 
For any $j\in D_{i}\setminus D_{i-1}$, we have
\[
 \inf_{M'}|\Delta_{M'}|^{\frac{1}{[M':\Q]^2}}\leq |\Delta_{F_{j}}|^{\frac{1}{[F_{j}:\Q]^2}}= |\Delta_{F_{j}}|^{\frac{1}{p_{j}^2}}\leq 3,
 \]
concluding the proof.
\end{proof}

\section{Proof of Theorem \ref{main-bld}}\label{proof-main-bld}

\noindent
This section is devoted to the proof of Theorem \ref{main-bld}, which is an easy consequence of the unramified case of
\cite[Theorem 2]{Fil},
which we quote here in the special case $K=\mathbb{Q}$:
\begin{theorem}[Fili]\label{Filthm}
Let $S$ be a finite set of prime numbers and, for every $p\in S$, \red{let}  $E_p/\mathbb{Q}_p$ be \red{a finite} Galois extension with ramification index $e_p$ and inertia degree $f_p$. 
If $L_p$ denotes the maximal Galois extension of $\Q$ inside $E_p$, then
\[
 \liminf_{\alpha\in\bigcap_{p\in S}L_p} h(\alpha)\leq \sum_{p\in S}  \frac{\log p}{e_p(p^{f_p}-1)}.
\] 
\end{theorem}


\begin{proof}[Proof of Theorem \ref{main-bld}]
Let $p_1,p_2,\dots$ be the sequence of prime numbers.
For each $i$ let $E_i$ be the unique unramified quadratic extension of $\mathbb{Q}_{p_i}$ and let $L_i$ be the maximal Galois extension of $\mathbb{Q}$ inside $E_i$.
Theorem \ref{Filthm} gives that, \red{for $n\geq 1$}
\[
  \liminf_{\alpha\in\bigcap_{i=1}^n L_i} h(\alpha)\leq\sum_{i=1}^n \frac{\log(p_i)}{p_i^{2}-1}<\sum_{k=1}^\infty \frac{\log(k)}{k^2}=:T<\infty.
\]
\red{So} for $n\geq 1$, we can iteratively pick algebraic elements $x_n\in\bigcap_{i=1}^n L_i\smallsetminus\{x_1,\dots,x_{n-1}\}$ with $h(x_n)<T$.
Then the Galois closure $L$ of the field $\mathbb{Q}(x_1,x_2,\dots)$ satisfies the claim.

Indeed,
clearly $L$ \red{does} not satisfy property (N), as $\{x_1,x_2,\dots\}$ is an infinite set of elements of $L$ of height at most $T$.
Furthermore, \red{for every $i\geq 1$}, the local \red{degree} of $L$ at $p_i$ \red{is} finite, as
$$
 [LE_i:E_i] \leq [LL_i:L_i]=[L_i(x_1,\dots,x_{i-1}):L_i]<\infty
$$
and therefore also 
$$
 [L\mathbb{Q}_{p_i}:\mathbb{Q}_{p_i}]\leq[LE_i:\mathbb{Q}_{p_i}]=[LE_i:E_i]\cdot [E_i:\mathbb{Q}_{p_i}]=[LE_i:E_i]\cdot 2<\infty.
$$ 
\end{proof}

\appendix

\section{Proof of Theorem \ref{sha-split}}
\label{shafarevich}

\noindent
Let $G$ be a finite solvable group, $K$ a number field and $S$ a finite set of primes of $K$.
We proceed, as in \cite[pag. 596]{NSW}, by induction on the order of the  group $G$. 
Suppose the statement is true for any solvable group of order strictly smaller than $|G|$.
By \cite[Propositions 9.6.8 and 9.6.9]{NSW} there exists a surjection $N\rtimes H\rightarrow G$ where $N$ is the (nilpotent) Fitting subgroup of $G$ and  $H$ is a proper subgroup of $G$. 
By the induction hypothesis there exists a Galois extension $L/K$ with Galois group $H$ which is totally split at all primes in $S$.
We now need to prove a refined version of \cite[Theorem 9.6.6]{NSW}, namely that every split embedding problem   
 \[1\rightarrow N\rightarrow N\rtimes H\rightarrow H\rightarrow 1\]
 for $L/K$ with finite nilpotent kernel $N$  has a proper solution which is totally split at all primes in $S$ \red{(see \cite[Definition 3.5.1]{NSW} for more details on embedding problems)}. 
Indeed, if $M/K$ is such a solution, then $\mathrm{Gal}(M/K)=N\rtimes H$ and therefore $M/K$ will contain a subextension $M'/K$ with $\mathrm{Gal}(M'/K)$ isomorphic to $G$.

We recall some notation: For a prime number $p$, we let $\mathcal{F}(n)$ be the free pro-$p$-$H$ operator group of rank $n$ (as defined in \red{\cite[Ch.~IX \S6 p.~578]{NSW}}). For $\nu = (i, j)$ with $i\geq j \geq 1$ we denote by $\mathcal{F}(n)^{(\nu)}$
the filtration of $\mathcal{F}(n)$ refining the descending $p$-central series, as defined in \cite[Ch.~III \S8]{NSW}. 
As every finite nilpotent group is a direct product of its $p$-Sylow subgroups and every finite $H$-operator $p$-group is a quotient of $\mathcal{F}(n)/\mathcal{F}(n)^{(\nu)}$ for some $n$ and $\nu$, cf.~\cite[Ch.~IX p.~584]{NSW}, our problem reduces to showing that for every prime $p$, every $n\geq 1$ and every $\nu=(i,j)$, the split embedding problem 
\begin{equation}\label{emb}
1\rightarrow \mathcal{F}(n)/\mathcal{F}(n)^{(\nu)}\rightarrow \mathcal{F}(n)/\mathcal{F}(n)^{(\nu)}\rtimes H\rightarrow H\rightarrow 1
\end{equation}
for $L/K$ has a solution $M/K$ which is totally split at primes in $S$.
This result can be proved using \cite[Theorem 9.6.7]{NSW}, which gives the existence of solutions to such split embedding problem with special local behavior.  
We adapt the proof of \cite[Theorem 6.2]{KM}, which is a generalisation of \cite[Theorem 9.6.7]{NSW} in the unramified case:

Write $S=S_0\cup S_\infty$ where $S_0=\{\mathfrak{p}_1,\ldots,\mathfrak{p}_m\}$ are the finite primes in $S$,
and let $p_i$ be the prime number generating the ideal $\mathfrak{p}_i\cap \mathbb{Z}$.
Let $\mu$ be the number of roots of unity in $K$
and choose a prime number $f>p\mu\cdot|H|\cdot[K:\Q]$.
As $f$ divides $\varphi(p_i^{f}-1)$, where $\varphi$ is the Euler totient function, 
and every prime divisor of $\varphi(p_i^{f}-1)$ is of the form $q$ or $q-1$ for a prime divisor $q$ of $p_i^f-1$,
there exists a prime number
$e_i\geq f$ that divides $p_i^{f}-1$.
The extension $M_i=K_{\mathfrak{p}_i}(\zeta_{p_i^{f}-1}, {p_i}^{1/e_i})$ of the completion $K_{\mathfrak{p}_i}$ is Galois with ramification index $e_i$,
and $G_i:={\rm Gal}(M_i/K_{\mathfrak{p}_i})$ is of order $e_if$ prime to $p$, $\mu$, $|H|$ and $[K:\Q]$. 
Let $\Gamma=\prod_{i=1}^m G_i$ and note that $\Gamma$ is solvable of order coprime to $\mu$.
Thus, by \cite[Corollary 2, p.156]{Neuk} (see also \cite[Theorem 9.5.5]{NSW})
there exists a Galois extension $L'/K$ with Galois group $\Gamma$ 
that has completion $L'_{\mathfrak{p}_i}=M_i$ at every $\mathfrak{p}_i$,
and completion $L'_\mathfrak{p}=K_\mathfrak{p}$ at every $\mathfrak{p}\in S_\infty$.
In particular, $L'$ is ramified at all $\mathfrak{p}\in S_0$ and totally split at all $\mathfrak{p}\in S_\infty$.
As $[L':K]=|\Gamma|$ is coprime to $[L:K]=|H|$,
$L'$ is linearly disjoint from $L$ over $K$, 
so that $\mathrm{Gal}(L'L/K)=H\times \Gamma$. 
By  \cite[Theorem 9.6.7]{NSW} we can find a solution $F/K$ to the embedding problem 
\begin{equation*}\label{emb-p}1\rightarrow \mathcal{F}(n)/\mathcal{F}(n)^{(\nu)}\rightarrow \mathcal{F}(n)/\mathcal{F}(n)^{(\nu)}\rtimes (H\times \Gamma) \rightarrow H\times \Gamma\rightarrow 1
\end{equation*}
for $L'L/K$ 
with the factor $\Gamma$ acting trivially, i.e. 
\[
 \mathcal{F}(n)/\mathcal{F}(n)^{(\nu)}\rtimes (H\times \Gamma)=(\mathcal{F}(n)/\mathcal{F}(n)^{(\nu)}\rtimes H)\times \Gamma,
\] 
which is totally split at all
infinite primes of $L'L$ as well as
at the finite primes of $L'L$ above the primes in $S_0$ (as they ramify in $L'L/K$).
The subfield $M$ of $F$ fixed by the factor $\Gamma$
is then a solution of the embedding problem \eqref{emb}.
The ramification index and the inertia degree of
each $\mathfrak{p}\in S_0$ in $F/K$ divide $|\Gamma|$,
and every $\mathfrak{p}\in S_\infty$ is totally split in $F/K$.
Thus, as $\mathrm{Gal}(M/K)=\mathcal{F}(n)/\mathcal{F}(n)^{(\nu)}\rtimes H$ has order prime to $|\Gamma|$,  
all $\mathfrak{p}\in S$ are totally split in $M$.

\section*{Acknowledgments} 
\noindent
The authors thank Lukas Pottmeyer for pointing out the results in \cite{Fil},
Emmanuel Kowalski for clarifications regarding \cite{IK},
Alexandra Shlapentokh for suggesting to include applications to undecidability,
and Fran\c{c}ois Legrand and Umberto Zannier for helpful comments and suggestions.
The first author's work has been funded by the ANR project Gardio 14-CE25-0015.
The second author was funded by the Deutsche Forschungsgemeinschaft (DFG) - 404427454.


\begin{thebibliography}{BMZ07}

\bibitem[AD00]{AD} F. Amoroso and R. Dvornicich, \emph{A lower bound for the height in abelian extensions}. Journal of Number Theory 80, 260--272 (2000).

\bibitem[ADZ14]{ADZ}F. Amoroso, S. David and U. Zannier, \emph{On fields with property (B).} Proceedings of the Amer.\ Math.\ Soc.\ 142(6), 1893--1910 (2014).

\bibitem[AZ00]{AZ1}F. Amoroso and U. Zannier, \emph{A relative Dobrowolski's lower bound
over abelian extensions.} Ann. Scuola Norm. Sup. Pisa Cl. Sci. (4) 29, no. 3, 711--727 (2000).

\bibitem[AZ10]{AZ2}F. Amoroso and U. Zannier, \emph{A uniform relative Dobrowolski's lower
bound over abelian extensions.} Bull. London Math. Soc.~42, no. 3, 489--498 (2010).


\bibitem[BZ01]{BZ} E. Bombieri and U. Zannier, 
\textit{A note on heights in certain infinite extensions of $\Q$}.
Rend. Mat. Acc. Lincei, 12,  5--14 (2001).

\bibitem[Che19]{Che19} S. Checcoli, \emph{A note on Galois groups and local degrees}.
Manuscripta Math., Vol. 159, Issue 1-2, 1--12 (2019). 


\bibitem[CW13]{CW} S. Checcoli and M. Widmer, \textit{On the Northcott property and other properties related to polynomial mappings}. Math. Proc. Cambridge Phil. Soc. 155(1), (2013).

\bibitem[DG12]{DG} P. D\`ebes and N. Ghazi, \emph{Galois covers and the Hilbert-Grunwald property}. Ann. Inst. Fourier 62(3), 989--1013 (2012).


\bibitem[Dim19]{Dim} V. Dimitrov, \emph{A proof of the Schinzel-Zassenhaus conjecture on polynomials}. arXiv:1912.12545 [math.NT] (2019).


\bibitem[DZ08]{DZ} R. Dvornicich and U. Zannier, \textit{On the properties of Northcott and Narkiewicz for fields of algebraic numbers}. Functiones et Approximatio 39:163--173, (2008).


\bibitem[Feh18]{Feh}
A. Fehm, 
\newblock \emph{Three counterexamples concerning the Northcott property of fields}. 
\newblock Atti Accad. Naz. Lincei Rend. Lincei Mat. Appl.~29(2), (2018).

\bibitem[Fil14]{Fil} 
P. Fili,
\newblock \emph{On the heights of totally $p$-adic numbers}. 
\newblock J. Th\'eor. Nombres Bordeaux 26(1):103--109 (2014).

\bibitem[Fre17]{Frey} L. Frey, \textit{Explicit Small Heights in Infinite Non-Abelian Extensions}.	arXiv:1712.04214 [math.NT] (2017).

\bibitem[Gal16]{Gal} A. Galateau, \textit{Small height in fields generated by singular moduli}. Proc. Amer. Math. Soc. 144 (2016).

\bibitem[Hab13]{Hab} P. Habegger, \emph{Small Height and Infinite Non-Abelian Extensions}, Duke Math. J. 162 no. 11, 1895--2076 (2013). 

\bibitem[IK04]{IK} H. Iwaniec and E. Kowalski, \emph{Analytic Number Theory}. Amer.\ Math.\ Soc. (2004).


\bibitem[KM04]{KM} J. Kl\"uners and G. Malle, \emph{Counting nilpotent Galois extensions}. J. reine angew. Math. 572, 1--26 (2004).

\bibitem[May92]{May92} D.~C.~Mayer, \emph{Multiplicities of dihedral discriminants}. Math. of Comp., Vol. 58, n. 198, 831--847 (1992).



\bibitem[Nar04]{Nark} W. Narkiewicz, \emph{Elementary and Analytical Theory of Algebraic Numbers}. Springer (2004).

\bibitem[Neu79]{Neuk} J. Neukirch, \emph{On solvable number fields}. Invent. math. 53, 135--164 (1979).

\bibitem[Neu99]{Neukirch} J.~Neukirch, \emph{Algebraic Number Theory}. Springer (1999).

\bibitem[NSW08]{NSW} J. Neukirch, A. Schmidt and K. Wingberg. \emph{Cohomology of number fields (Second Edition)}.
 Springer (2008).

\bibitem[Ple19]{Plessis} A. Plessis, \emph{
Minoration de la hauteur de Weil dans un compositum de corps de rayon}. 
J.\ Number Theory  205, 246--276 (2019).

\bibitem[Pot15]{Pot} L. Pottmeyer, \textit{Heights and totally $p$-adic numbers}. Acta Arith. 171(3), 277--291 (2015).

\bibitem[Pot16]{Pot2} L. Pottmeyer, \textit{A note on extensions of $\Q^{tr}$}. J.\ Th\'eor.\ Nombres Bordeaux 28(3), 735--742 (2016).

\bibitem[Sch73]{Sch} A. Schinzel, \emph{On the product of the conjugates outside the unit circle
of an algebraic number}. Acta Arith.~24 (1973), 385--399. Addendum, ibidem, 26, 329--361 (1973).

\bibitem[Shl18]{Shlapentokh}
A.~Shlapentokh, \emph{First-order decidability and definability of integers in infinite algebraic extensions of the rational numbers.}
Israel J.\ Math.\ 226, 579--633 (2018).





\bibitem[VV16]{VV} X.~Vidaux and C.~R.~Videla. \emph{A note on the Northcott property and undecidability}. 
Bull.\ London Math.\ Soc.\ 48, 58--62 (2016)

\bibitem[Vid00]{Videla} C.~R.~Videla. \emph{Definability of the ring of integers in pro-p Galois extensions of number fields.}
Israel J.\ Math.\ 118, 1--14 (2000).

\bibitem[Wal36]{wal} A. Walfisz, \emph{Zur additiven Zahlentheorie. II} Mathematische Zeitschrift (in German). 40(1), 592--607 (1936).

\bibitem[Wid11]{Wid} M. Widmer, \textit{On certain infinite extensions of the rationals with Northcott property}. Monatsh.
Math. 162(3), 341--353 (2011).


\end{thebibliography}
\end{document}